\documentclass[12pt]{amsart}
\usepackage{amsmath,amsthm,amsfonts,amscd,amssymb,eucal,latexsym,mathrsfs, appendix, yhmath, setspace}
\usepackage[numbers,sort&compress]{natbib}
\usepackage[all,cmtip]{xy}
\usepackage{enumerate}
\usepackage{mathptmx}
\usepackage{mathrsfs}
\usepackage{tikz}
\usepackage{verbatim}
\usepackage{url}
\usepackage[all]{xy}
\usepackage{xcolor}
\usepackage[colorlinks=true,citecolor=blue]{hyperref}
\usepackage{multicol}
\theoremstyle{plain}

\newtheorem{theorem}{Theorem}[section]
\newtheorem{corollary}[theorem]{Corollary}
\newtheorem{lemma}[theorem]{Lemma}

\newtheorem{claim}{Claim}
\theoremstyle{definition}

\newcommand{\supp}{{\rm supp}}

\newcommand{\N}{\mathbb{N}}

\newcommand{\mZ}{\mathbb Z}

\newcommand{\ep}{\varepsilon}

\allowdisplaybreaks

\begin{document}

\title
[Minimal subsystems of given mean dimension in Bernstein spaces]
{Minimal subsystems of given mean dimension\\ in Bernstein spaces}

\author[J.~Zhao]{Jianjie Zhao}
\address[J.~Zhao]{School of Mathematics, Hangzhou Normal University, Hangzhou, Zhejiang, 311121, P.R. China }
\email{zjianjie@hznu.edu.cn}

\date{\today}


\begin{abstract}
 In this paper, we study the shift on the space of uniformly bounded continuous functions band-limited in a given compact interval with the standard topology of tempered distributions. We give a constructive proof of the existence of minimal subsystems with any given mean dimension strictly less than twice its band-width. A version of real-valued function spaces is considered as well.
\end{abstract}

\maketitle

\section{Introduction}  
Mean dimension is a topological invariant of dynamical systems introduced by Gromov \cite{Gromov-99} in 1999, which counts the number of parameters per second for describing a dynamical system just as topological entropy counts the number of bits per second for describing a dynamical system. Lindenstrauss and Weiss \cite{LW-00} firstly established a deep application of mean dimension by solving Auslander's open problem in the negative: Whether every minimal dynamical system can be embedded\footnote{ We say that a dynamical system can be \textbf{embedded} into another dynamical system if the first system can be regarded as a subsystem of the second one.} into the full shift on the Hilbert cube $([0, 1]^{\mathbb{Z}}),\sigma)$, where 
$$\sigma:[0, 1]^{\mathbb{Z}}\to[0, 1]^{\mathbb{Z}},\;\;\;(x_{n})_{n\in\mathbb{Z}}\mapsto(x_{n+1})_{n\in\mathbb{Z}}.$$ 
Soon after, Lindenstrauss \cite{Lindenstrauss-99} proved that any minimal dynamical system of mean dimension less than $c$ for $c={1}/{36}$ can be embedded into the full shift on the Hilbert cube. This is a fairly amazing result. A natural question that was also posed by Lindenstrauss in \cite{Lindenstrauss-99} is what the optimal constant $c$ is. In 2014, Lindenstauss and Tsukamoto \cite{LT-14} constructed a minimal dynamical system of mean dimension equal to $1/2$ that cannot be embedded into the full shift on the Hilbert cube. This implies that the optimal constant $c$ will not exceed $1/2$. In 2020, Gutman and Tsukamoto \cite{GT-20} confirmed that the value $c=1/2$ is optimal, that is, any minimal dynamical system of mean dimension less than $1/2$ can be embedded into the full shift on the Hilbert cube. The statement of this profound result lies purely in abstract topological dynamics but the approach in \cite{GT-20} involved signal analysis (or Fourier analysis) and complex analysis theory. More precisely, they encoded a given dynamical system into time-continuous signals whose frequencies are limited in a fixed interval (i.e. band-limited) by interpolation and then converted band-limited continuous signals into time-discrete ones by sampling. A similar approach was employed in \cite{GQT19} to solve the more complicated embedding problem of multi-parameter actions as well. This motivates us to study the spaces of band-limited continuous signals, i.e. Bernstein spaces. Let us first give the definitions.

A \textbf{rapidly decreasing function} is an infinitely differentiable function $f$ on $\mathbb{R}$ 
satisfying
$$ \lim\limits_{|t|\to \infty} t^nf^{(j)}(t)=0, \ \ \ \  \forall n, j\in \mathbb{N}.$$
A \textbf{tempered distribution} on $\mathbb{R}$ is a continuous linear functional on the space of all rapidly decreasing functions equipped with the topology given by a family of seminorms as follows:
$$ \left\| {f} \right\|_{j, n}=\sup_{t\in \mathbb{R}}|t^nf^{(j)}(t) | \ \ \ \ (j, n\in \mathbb{N}). $$ 
For a rapidly decreasing function $f:\mathbb{R} \to \mathbb{C}$ the Fourier transform
of $f$ is given by
$$  \mathcal{F}(f)( \xi )=\int_{-\infty}^{+\infty} e^{-2\pi \sqrt{-1}t \xi}f(t) dt, \ \ 
\overline{\mathcal{F}} (f)(t )=\int_{-\infty}^{+\infty} e^{ 2\pi \sqrt{-1}t \xi}f(\xi) d\xi.$$
The operators $\mathcal{F}$ and $\overline{\mathcal{F}}$ can be extended to tempered distributions in a standard way.
Let $I$ be a compact subset of $\mathbb{R}$. 
A bounded continuous function $f: \mathbb{R} \to \mathbb{C}$ is called {\bf band-limited in $I$} if $\supp(\mathcal{F}(f)) \subset I$,
meaning that $\left\langle {{\mathcal{F}(f), g  }} \right\rangle=0$ for all rapidly decreasing functions 
$g: \mathbb{R} \to \mathbb{C}$ with $\supp(g) \cap I=\emptyset$. We remark here that $\supp(\mathcal{F}(\bar{f}))=-\supp(\mathcal{F}(f)) $.
We denote by $\mathcal{B}^{\mathbb{C}}(I)$ (resp. $\mathcal{B}(I)$) the set of continuous functions $f:\mathbb{R} \to \mathbb{C}$ (resp. $f:\mathbb{R} \to \mathbb{R}$) band-limited in $I$ with $\left\| {f} \right\|_{L^{\infty}(\mathbb{R})}\le 1$ and call it a \textbf{Bernstein space}. For two functions $f, g\in \mathcal{B}^{\mathbb{C}}(I)$ (resp. $\mathcal{B}(I)$), we define   
 \begin{equation*}
D(f, g)=\sum_{n=1}^{\infty} \frac{ \left\| {f-g} \right\|_{L^{\infty}([-n, n])}  }{2^n}.
\end{equation*}
An important and nontrivial fact \cite[Lemma 2.3]{GT-20} \cite[Chapter 7, Section 4]{Sch-66} is that both $\mathcal{B}^{\mathbb{C}}(I) $ and $\mathcal{B}(I) $ are compact metric spaces with respect to the distance $D$ which coincides with the standard topology of tempered distributions. We define 
$$\sigma\footnote{We use the same notation $\sigma$ as in the dynamical system $([0, 1]^{\mathbb{Z}}),\sigma)$ above if there is no misleading.}: \mathcal{B}^{\mathbb{C}}(I) \to \mathcal{B}^{\mathbb{C}}(I), \ \ \ g(\cdot)\mapsto g(\cdot+1)$$
and 
$$\sigma: \mathcal{B}(I) \to \mathcal{B}(I), \ \ \ g(\cdot)\mapsto g(\cdot+1).$$
Clearly, both $\mathcal{B}^{\mathbb{C}}(I) $ and $ \mathcal{B}(I)$ are $\sigma$-invariant.
We consider the dynamical system $(\mathcal{B}^{\mathbb{C}}(I), \sigma)$ (resp. $(\mathcal{B}(I), \sigma )$)
and call this system \textbf{the shift on} $\mathcal{B}^{\mathbb{C}}(I)$ (resp. $\mathcal{B}(I)$).
\bigskip 

As the most important space of time-discrete signals, the system $([0,1]^{\mathbb{Z}},\sigma)$ has been studied intensively and deeply. One of hot topics is about the range of mean dimension of all minimal subshifts. As we mentioned above, the result in \cite{GT-20} implies that any value not greater than $1/2$ is possible. In 2017, Dou \cite{Dou17} constructed minimal subshifts of $([0,1]^{\mathbb{Z}},\sigma)$ with arbitrary mean dimension strictly less than $1$. It remains open that whether there exists a minimal subshift  of mean dimension $1$. Recently, Jin and Qiao \cite{JQ-21} answered this question affirmatively. It is natural to ask a similar question about the space of time-continuous signals, i.e. what the range of mean dimension of all minimal subsystems of a given Bernstein space is. In the current paper, we prove the following main result. 

\begin{theorem}[Main theorem 1]\label{Main Theorem 1}
	Let $a,b,c\in \mathbb{R}$ with $a<b$ and $c>0$. Then we have the following:
	\begin{enumerate}
		\item \label{1} for any $0\le s < 2(b-a)$, there exists a minimal subsystem $(X, \sigma)$ of $(\mathcal{B}^{\mathbb{C}}([a,b]),\sigma)$ of mean dimension $s$.
		
		\item \label{2} for any $0\le t< 2c$, there exists a minimal subsystem $(X, \sigma)$ of $(\mathcal{B}([-c,c]),\sigma)$ of mean dimension $t$.
	\end{enumerate}
\end{theorem} 

We remark here that Jin, Qiao and Tu systematically studied the Bernstein spaces in \cite{JQT-21}; in particular, they obtained that the mean dimension of $(\mathcal{B}^{\mathbb{C}}([a,b]),\sigma)$ and $(\mathcal{B}([-c,c]),\sigma)$ is $2(b-a)$ and $2c$ respectively for all real numbers $a\leq b$ and $c\geq0$. The reader may wonder what the difference is between our result (Theorem \ref{Main Theorem 1}) and Dou \cite{Dou17} and whether they are equivalent. Applying the sampling theorem to Theorem \ref{Main Theorem 1}, we get a stronger version of the result in \cite{Dou17}. 
\begin{theorem}[Main theorem 2] \label{Main Theorem 2}
	Let $a\in\mathbb{R}$ with $0\le a <1$. Then for any $0\le r < a$, there exists a minimal system of mean dimension $r$ which can be embedded into both $(\mathcal{B}([-\frac{a}{2},\frac{a}{2}]),\sigma)$ and $([0, 1]^{\mathbb{Z}},\sigma)$ 
\end{theorem}

The paper is organized as follows. In Section 2, we gather definitions that will be needed in the article. In Section 3, we give the constructive proof of the main theorems.

\medskip

\noindent {\bf Acknowledgments.}
The authors would like to thank Lei Jin and Yixiao Qiao for helpful discussions and remarks.

\section{Preliminaries}  

In the whole paper, we denote by $\N$ and $\mZ$ the set of all positive integers and all integers respectively, and we write the cardinality of $F$ by $|F|$ for any finite subset $F$ of $\N$. 

\subsection{Full shifts}
By a \textbf{topological dynamical system}, we mean a pair $(X,T)$ where $X$ is a compact metric space and $T: X\to X$ is a homeomorphism. A class of canonical and significant topological dynamical systems are full shifts. Let $K$ be a compact metric space with metric $d$ and $K^{\mathbb{Z}}$ the countable product of $K$ with the product topology. The \textbf{full shift over $K$} is $(K^{\mathbb{Z}}, \sigma_{K})$, where 
$$\sigma_{K}:K^{\mathbb{Z}}\to K^{\mathbb{Z}},\;\;\;(x_{n})_{n\in\mathbb{Z}}\mapsto(x_{n+1})_{n\in\mathbb{Z}}.$$
We fix a compatible metric $D_1$ on $K^{\mathbb{Z}}$ as follows:
$$D_1(x, y)=\sum_{n\in \mathbb{Z}} \frac{d(x_n,y_n)}{2^{|n|}}, \;\; x=(x_n)_{n\in \mathbb{Z}},\; y=(y_n)_{n\in \mathbb{Z}}\in K^{\mathbb{Z}}.$$
 
For $x=(x_n)_{n\in \mathbb{Z}}\in K^{\mathbb{Z}}$ and two integers $m\le m'$,
we set 
$$x|_{m}^{m'}=x|_{[m, m']}=(x_n)_{m\le n\le m'}.$$

\subsection{Mean dimension}
Let $X$ and $P$ be two compact metrizable spaces and $d$ a compatible metric on $X$.
For $\ep>0$, 
a continuous mapping $f:X\to P$ is called an $\ep$-\textbf{embedding} with respect to $d$ if it satisfies
$$f(x)=f(x') \Rightarrow  d(x, x')<\ep, \ \ \forall x, x'\in X.$$ 
We define Widim$_\ep(X, d)$ as the minimum topological dimension dim$(P)$ of a compact metrizable space $P$ which admits an 
$\ep$-embedding $f:X\to P$ with respect to $d$. We state a very practical result in the calculation of mean dimension of subshifts.

\begin{lemma}[{\cite[Lemma 1.1.1]{Gromov-99}}]\label{Widim of [0,1]^l}
	For any $\ep>0$ and any positive integer $l$
	$$\emph{Widim}_{\ep}([0, 1]^{l},||\cdot ||_{\infty})=l.$$
	Here $||\cdot ||_{\infty}$ is the compatible metric on $ [0, 1]^{l}$ defined by
	$$||(a_n)_{n=0}^{l-1}-(b_n)_{n=0}^{l-1} ||_{\infty}=\max_{0\le n\le l-1}|a_n-b_n|,\;\;\;(a_n)_{n=0}^{l-1}, (b_n)_{n=0}^{l-1}\in [0, 1]^{l}.$$
\end{lemma} 

Let $(X, T)$ be a topological dynamical system with a compatible metric $d$ on $X$.
For every positive integer $n$, we define a compatible metric $d_n$ on $X$ as follows:
$$d_n(x, x')=\max_{0\le i\le n-1}d(T^ix, T^ix'), \ \  \forall x, x'\in X.$$ 
The \textbf{mean dimension} of $(X, T)$ is defined by 
$$ \text{mdim}(X, T)=\lim\limits_{\ep\to 0} \lim\limits_{n\to \infty} \frac{ \text{Widim}_{\ep}(X, d_n) }{n}. $$
It is well known that the limits above definition always exist, and the value mdim$(X, T)$ is independent of the choice of a 
compatible metric $d$ on $X$. 

\subsection{Sampling and interpolation}
The following theorems are key tools in our proof, which allow us to convert continuous signals to discrete ones (Theorem \ref{thm:sampling}) and vice verse (Theorem \ref{thm:interpolation}) as we mentioned in the introduction.

\begin{theorem}[{{\cite[Lemma 2.4]{GT-20}}, {\cite[Lemma 4.2]{JQT-21}}}]\label{thm:sampling}
	Suppose that two positive real numbers $a$ and $d$ satisfy $2ad < 1$ and $f\in \mathcal{B}([-a, a])$. If $f(dn)=0$ for all $n\in \mathbb{Z}$, then $f\equiv 0$.
\end{theorem}

\begin{theorem}[{{\cite[Section 5]{GT-20}}, {\cite[Lemma 4.3]{JQT-21}}}]\label{thm:interpolation}
	For every $\ep>0$, there exists a rapidly decreasing function $f: \mathbb{R} \to \mathbb{R}$ band-limited in $[-(1+\ep)/2, (1+\ep)/2]$ such that $f(0)=1$ and $f(n)=0$ for all nonzero $n\in \mathbb{Z}$.
\end{theorem}

Apply the coordinate transformation to Theorem \ref{thm:interpolation}, we obtain that

\begin{corollary}\label{cor:interpolation}
	For any $u, v\in \mathbb{N}, t_1< t_2\in \mathbb{R}$ with $\frac{v}{u}+1<t_2-t_1$,	there exists a rapidly decreasing function 
	$f: \mathbb{R}\to \mathbb{R}$ band-limited in $[t_1,t_2]$ such that 
	$$f(0)=1 \;\;\;\text{and}\;\;\; f(\frac{u}{v}n)=0 \;\;\;\text{for all}\;\;n\in\mathbb{Z}\setminus\{0\}.$$	
\end{corollary}
\begin{proof}
	For $\frac{u}{v}:=\ep$, by Theorem \ref{thm:interpolation}, there exists a rapidly decreasing function $g: \mathbb{R}\to \mathbb{R}$
	band-limited in $[-\frac{1+\ep}{2},\frac{1+\ep}{2}]$ such that 
	$$g(0)=1 \;\;\;\text{and}\;\;\; g(n)=0 \;\;\;\text{for all}\;\;n\in\mathbb{Z}\setminus\{0\}.$$	
	Put 
	$$f(x)=g(\frac{v}{u}x)\cdot e^{2\pi \sqrt{-1}xx_0},   \ \    \forall  x\in \mathbb{R},$$
	where $x_0=t_1+(1+\frac{v}{u})/2$. 
	Then $f$ is what we want.

\end{proof}

\section{Proofs of Theorems \ref{Main Theorem 1} and \ref{Main Theorem 2}}
\begin{proof}[Proof of Theorem \ref{Main Theorem 1} (1)]
Fix $p,q\in\mathbb{N}$ and $\epsilon_{0},c\in\mathbb{R}$ with 
$$\frac s2<\frac{q}{p}<b-a,\;\;\;0<\epsilon_{0}<b-a,\;\;\;a<c<a+\frac{\epsilon_{0}}{2},\;$$
$$ \frac{q}{p}+\epsilon_0+1<b-a \;\;\; \text{and}\;\;\;cp\in\mathbb{N}.$$
Put 
$$r=(\frac{s}{2})/(\frac{q}{p}).$$
Clearly, $0<r<1$. Put $K=([0,1]^{2})^q$.  
We equip $K$ with the metric $||\cdot||_{\infty}$ defined by
$$||(a_n)_{n=0}^{q-1}-(b_n)_{n=0}^{q-1} ||_{\infty}=\max_{0\le n\le q-1} |a_n -b_n |_2, $$
where
$$|a_n-b_n|_2=\max \{|a_n^1-b_n^1|, |a_n^2-b_n^2| \} ,$$
for any 
$(a_n)_{n=0}^{q-1},(b_n)_{n=0}^{q-1}\in K$ with $ a_n=(a_n^1, a_n^2), b_n=( b_n^1, b_n^2)\in [0, 1]^2. $

\vspace{0.2cm}
For each $m\in\mathbb{N}$, we take a $\frac{1}{m}$-dense subset $P_m$ of $K$. 

We take a symbol $\ast\notin K$ and set $P=\{\ast\}\cup K$. 
For $N\in \mathbb{N}$, $x=(x_n)_{n=0}^{N-1}\in P^N$,
we set
$$x(N,\ast)=\left\{n\in[0,N-1]:x_{n}=\ast\right\}.$$

Now we are going to construct inductively a minimal subsystem $(X,\sigma)$ of the Bernstein space
$(\mathcal{B}^{\mathbb{C}}([a,b]),\sigma)$ of mean dimension $s$.

\vspace{0.2cm} 
\textbf{Part 1. The construction of $(X,\sigma)$.} 

\medskip

We need to construct a subshift $(Y,\sigma)$ of $(K^{\mathbb{Z}}, \sigma)$ firstly as a bridge to $(X,\sigma)$.
\vspace{0.2cm} 

Step $0$. Take $n_0=N_0=1$. Let $B_0=K$ and $Y_0=K^{\mathbb{Z}}$.

\vspace{0.2cm} 

Step $1$. We choose $n_1=N_{1}\in \mathbb{N} $ sufficiently large so that there is 
$x^{(1)}\in P^{N_1}$ such that

$$r < \frac{ |x^{(1)}(N_1, \ast)|}{N_1}\le r+\frac{1}{N_1}.$$

Put 
$$B_1=\left\{(a_n)_{n=0}^{N_1-1}\in K^{N_1}: a_n=x_n^{(1)}, \ \  \forall n\in [0, N_1-1] \backslash x^{(1)}(N_1,\ast)\right\}$$
and 
$$Y_1=\left\{(a_n)_{n\in\mathbb{Z}} \in K^{\mathbb{Z}}:\ \exists \ m \in [0, N_1-1]
\;\text{s.t.}\;(a_{m+i+jN_1})_{i=0}^{N_1-1}\in B_1, \ \ \forall  j\in\mathbb{Z}\right\}.$$

\vspace{0.2cm} 

Step $2$. We choose $n_2\in\mathbb{N}$ sufficiently large such that
$$ \frac{n_2\cdot |x^{(1)}(N_1, \ast)|-|P_1|^{|x^{(1)}(N_1, \ast)|} \cdot N_1}{n_2\cdot N_1}> r.$$
Clearly, we have $n_2>|P_1|^{|x^{(1)}(N_1, \ast)|}$.

\vspace{0.2cm}

Set $N_{2}=n_{2}\cdot N_{1}$. Pick $x^{(2)}=(x_{n}^{(2)})_{n=0}^{N_{2}-1}\in P^{N_2}$ satisfying that
\begin{enumerate}
	\item [(i)] for all $0\le j\le n_2-|P_1|^{|x^{(1)}(N_1, \ast)|}-1$, 
	$$(x^{(2)}_{jN_1}, x^{(2)}_{jN_1+1},\cdots, x^{(2)}_{(j+1)N_1-1})=x^{(1)};$$
	\item [(ii)] for all $n_2-|P_1|^{|x^{(1)}(N_1, \ast)|}\le j\le n_2-1$ and $l\in [0, N_1-1]\backslash x^{(1)}(N_1, \ast)$, $$x^{(2)}_{jN_1+l}=x_l^{(1)};$$
	\item [(iii)] for all $n_2-|P_1|^{|x^{(1)}(N_1, \ast)|}\le j\le n_2-1$ and $l\in x^{(1)}(N_1, \ast)$, 
	$$x^{(2)}_{jN_1+l}\in P_1;$$
	\item [(iv)] $(x^{(2)}_{jN_1}, x^{(2)}_{jN_1+1}, \cdots, x^{(2)}_{(j+1)N_1-1})$ ($n_2-|P_1|^{|x^{(1)}(N_1, \ast)|}\le j\le n_2-1$) are pairwise distinct. 
	
\end{enumerate}

After replacing an appropriate number of $\ast$'s by $0$ from the last coordinate of $x^{(2)}$ (still denoted by $x^{(2)}$), we may require that
$$r< \frac{|x^{(2)}(N_2, \ast)|}{N_2}\le r+\frac{1}{N_2}.$$

Put 
$$B_2=\left\{(a_n)_{n=0}^{N_2-1}\in K^{N_2}: a_n=x_n^{(2)},\ \ \forall  n \in [0, N_2-1]\backslash x^{(2)}(N_2, \ast)\right\}$$
and
$$Y_2=\left\{(a_n)_{n\in\mathbb{Z}} \in K^{\mathbb{Z}}: \ \exists \ m \in [0, N_2-1] \
\text{s.t.}\ (a_{m+i+jN_2})_{i=0}^{N_{2}-1}\in B_2, \  \ \forall j\in\mathbb{Z}\right\}.$$

\vspace{0.2cm} 

Assume that $x^{(k-1)}$, $B_{k-1}$ and $Y_{k-1}$ ($k\geq 2$) have been constructed. Now we proceed to construct $x^{(k)}$, $B_{k}$ and $Y_{k}$.

\vspace{0.2cm} 

Step $k$. We choose $n_k\in\mathbb{N}$ sufficiently large such that 
$$\frac{n_k \cdot |x^{(k-1)}(N_{k-1}, \ast)|- |P_{k-1}|^{|x^{(k-1)}(N_{k-1}, \ast)|}\cdot N_{k-1}}{n_{k}\cdot N_{k-1}}>r.$$
Clearly, we have $n_k>|P_{k-1}|^{|x^{(k-1)}(N_{k-1}, \ast)|}$.

Set $N_{k}=n_{k}\cdot N_{k-1}$. Pick $x^{(k)}=(x_{n}^{(k)})_{n=0}^{N_{k}-1}\in P^{N_{k}}$ satisfying that
\begin{enumerate}
	\item [(i)] for all $0\le j\le n_k-|P_{k-1}|^{|x^{(k-1)}(N_{k-1}, \ast)|}-1$, 
	$$(x^{(k)}_{jN_{k-1}}, x^{(k)}_{jN_{k-1}+1}, \cdots, x^{(k)}_{(j+1)N_{k-1}-1})=x^{(k-1)};$$
	\item [(ii)] for all $n_k-|P_{k-1}|^{|x^{(k-1)}(N_{k-1}, \ast)|}\le j\le n_k-1$ and $l\in [0, N_k-1]\backslash x^{(k-1)}(N_{k-1},\ast)$,
	$$x^{(k)}_{j N_{k-1}+l}=x_l^{(k-1)};$$
	
	\item [(iii)] for all $n_k-|P_{k-1}|^{|x^{(k-1)}(N_{k-1}, \ast)|}\le j\le n_k-1$ and $l\in x^{(k-1)}(N_{k-1},\ast)$, 
	$$x^{(k)}_{jN_{k-1}+l}\in P_{k-1};$$
	
	\item [(iv)] $(x^{(k)}_{jN_{k-1}}, x^{(k)}_{jN_{k-1}+1}, \cdots, x^{(k)}_{(j+1)N_{k-1}-1})$ ($n_k-|P_{k-1}|^{|x^{(k-1)}(N_{k-1}, \ast)|}\le j\le n_k-1$) are pairwise distinct.
\end{enumerate}

After replacing an appropriate number of $\ast$'s by $0$ from the last coordinate of $x^{(k)}$ (still denoted by $x^{(k)}$), we may require    
\begin{equation}\label{eq:proportion_xk}
r<\frac{ |x^{(k)}(N_{k},\ast)|}{N_{k}}\le r+ \frac{1}{N_{k}}.
\end{equation}
Put
$$B_k=\left\{(a_n)_{n=0}^{N_{k}-1}\in K^{N_{k}}: a_n=x_n^{(k)}, \ \ \forall n\in [0, N_k-1]\backslash
x^{(k)}(N_{k},\ast)\right\}$$
and
$$Y_k=\left\{(a_n)_{n\in\mathbb{Z}}\in K^{\mathbb{Z}}:\ \exists\ m\in [0, N_k-1] \\
\ \text{s.t.}\ (a_{m+i+jN_{k}})_{i=0}^{N_{k}-1}\in B_k,\ \forall  j\in\mathbb{Z}\right\}.$$

\vspace{0.2cm}

Note that $Y_k$ is a nonempty closed $\sigma$-invariant subset of $K^{\mathbb{Z}}$ 
and that $Y_{k+1}\subset Y_k$ for all $k\ge 1$. Set 
$$Y=\bigcap_{k=1}^{\infty}Y_k.$$ 
Thus $(Y,\sigma)$ is a nonempty closed $\sigma$-invariant subset of $K^{\mathbb{Z}}$, i.e.,
$(Y,\sigma)$ is a subshift of $(K^{\mathbb{Z}}, \sigma)$. 

\vspace{0.2cm} 

To construct the targeted system $(X,\sigma)$, we need to convert discrete signals into continuous ones. In order to do this, we take a rapidly decreasing function $f: \mathbb{R}\to\mathbb{C}$ band-limited in $[a+\epsilon_{0},b]$ with 
$$f(0)=1\;\;\;\text{and}\;\;\;f(\frac{p}{q}n)=0\;(n\in\mathbb{Z}\setminus\{0\})$$
by Corollary \ref{cor:interpolation}.
Moreover, for some positive constant $C_1$, we have
$$|f(x)|\le \frac{C_1}{1+x^2}, \ \ \forall x\in \mathbb{R}.$$

Let $\mathbb{Z}_p=\mathbb{Z}/{p\mathbb{Z}}=\{0,1,\dots,p-1\}$ be the $p$-cyclic group with the discrete topology. Consider the product spaces $\mathcal{B}^{\mathbb{C}}([a+\ep_0, b]) \times\mathbb{Z}_p$ and $Y\times \mathbb{Z}_p$. 
We define 
$$T:\mathcal{B}^{\mathbb{C}}([a+\ep_0, b]) \times\mathbb{Z}_p\to\mathcal{B}^{\mathbb{C}}([a+\ep_0, b]) \times\mathbb{Z}_p$$ 
and $$S:Y\times\mathbb{Z}_p\to Y\times\mathbb{Z}_p$$ 
as follows: for all $g\in\mathcal{B}^{\mathbb{C}}([a+\ep_0, b])$ and $t\in\mathbb{R}$,
$$T(g(t), i )=\begin{cases}
(g(t+1), i+1)   &0\le i\le p-2\\
(g(t+1), 0)  &i=p-1
\end{cases}$$
and for all $y\in Y$,
\begin{equation}\label{def:S}
S(y,i)=
\begin{cases}
(y,i+1)   &0\le i\leq p-2 \\
(\sigma(y), 0)  &i=p-1
\end{cases}
\end{equation}
respectively.
Clearly, $T$ and $S$ are homeomorphisms (We consider the product topology of both $\mathcal{B}^{\mathbb{C}}([a+\ep_0, b]) \times\mathbb{Z}_p$ and $Y\times\mathbb{Z}_p $).

Define a mapping $F$ from $Y\times\mathbb{Z}_p$ to $\mathcal{B}^{\mathbb{C}}([a+\ep_0, b]) \times \mathbb{Z}_p$:
$$Y\times\mathbb{Z}_p\to\mathcal{B}^{\mathbb{C}}([a+\ep_0, b]) \times \mathbb{Z}_p$$
$$((a_n)_{n\in \mathbb{Z}}, i)\mapsto
\left(\frac{1}{C}\sum\limits_{n\in \mathbb{Z}}\sum_{j=0}^{q-1}\big(a_n^{j,1}+a_n^{j,2}\sqrt{-1}\big)f\big(x-\frac{p}{q}(nq+j)+i\big), i\right),$$
where $a_{n}=(a_{n}^{j})_{j=0}^{q-1}\in K=([0,1]^{2})^q$, $a_{n}^{j}=(a_{n}^{j,1},a_{n}^{j,2})\in[0,1]^{2}$ and 
$$C=\max_{i\in \mathbb{Z}_p}\max_{x\in\mathbb{R}}\sum_{n\in\mathbb{Z}}\sum_{j=0}^{q-1}\frac{C_{1}}{1+(x- \frac{p}{q}(n+j)+i)^{2}}<\infty.$$
Define a mapping $G$ from $\mathcal{B}^{\mathbb{C}}([a+\ep_0, b]) \times\mathbb{Z}_p$: for all $(g,i)\in\mathcal{B}^{\mathbb{C}}([a+\ep_0, b]) \times\mathbb{Z}_p$,
$$G(g(x), i)=\frac{g(x)+H(i)(x)}{2} \ \ \text{for all} \ \ x\in \mathbb{R},$$
where $H: \mathbb{Z}_p \to \mathcal{B}^{\mathbb{C}}([a, a+\frac{\ep_0}{2}])$ is defined by
$i\mapsto H(i)(x)=e^{-2\pi \sqrt{-1}c(x+i)}$. 
Clearly, $H$ is continuous and injective.

Set $$X=G\circ F(Y\times \mathbb{Z}_p).$$

\bigskip

\textbf{Part 2. $(X,\sigma)$ is a subsystem of $(\mathcal{B}^{\mathbb{C}}([a, b]), \sigma)$ which is topologically conjugate to $(Y\times Z_{p},S)$.}

\medskip

It suffices to show that $F$ and $G$ are embeddings from $(Y\times\mathbb{Z}_p,S)$ to $(\mathcal{B}^{\mathbb{C}}([a+\ep_0, b]) \times\mathbb{Z}_p, T)$ and from $(\mathcal{B}^{\mathbb{C}}([a+\ep_0, b]) \times\mathbb{Z}_p, T)$ to $(\mathcal{B}^{\mathbb{C}}([a, b]),\sigma)$, respectively.

\medskip

We first verify the embeddability of $F$. To show the continuity of $F$, we consider a compatible metric $\rho$  on $Y\times Z_{p}$: for any $x,y\in Y$ and $i,j\in \mathbb{Z}_p$, 
$$\rho((x,i),(y,j))=\begin{cases}
D_1(x,y)   &i=j\\
2 &i\neq j
\end{cases},$$
and a compatible metric ${\bf d}$ on $\mathcal{B}^{\mathbb{C}}([a+\ep_0, b])\times \mathbb{Z}_p$: for any $g_1, g_2\in \mathcal{B}^{\mathbb{C}}([a+\ep_0, b])$ and $ i, j\in \mathbb{Z}_p$,
$${\bf d}((g_1, i), (g_2, j) )=\begin{cases}
D(g_1, g_2)   &i=j\\
\ \ \ \ 2  &i\neq j
\end{cases}.$$

Fix $(a, i)=((a_n)_{n\in \mathbb{Z}}, i)\in Y\times Z_{p}$.
Take a sequence 
$$\{(a^{(m)}, i_m)=((a^{(m)}_n)_{n\in \mathbb{Z}}, i_m)\}_{m\in \mathbb{N}}\subset Y\times \mathbb{Z}_{p}$$
with $ \rho((a^{(m)}, i_m), (a, i))\to 0 $ as $m\to \infty$. Without loss of generality, we may assume $i_m=i$ for all $m\in \mathbb{N}$. Put
$$F((a^{(m)}, i))=(g^{(m)}, i) \ \ \text{and}\ \ F((a, i))=(g, i)$$
for some $g^{(m)},g\in\mathcal{B}^{\mathbb{C}}([a+\ep_0,b])$.

Fix $N\in\mathbb{N}$. For any $\delta>0$, there exists $L\in \mathbb{N}$ sufficiently large such that 
$$\frac{1}{C}\sum_{|n|>L}\sum_{j=0}^{q-1} \frac{C_1}{1+(x- \frac{p}{q}(n+j)+i)^2}< \frac{\delta}{2q}\;\;\; (\forall i\in\mathbb{Z}_p, \;x\in[-N,N]).$$
For such $L\in \mathbb{N}$, we take $M\in \mathbb{N}$ sufficiently large satisfying 
$$||a^{(m)}_n-a_n||_{{\infty}}<\frac{\delta}{2} \ \ (\forall n\in [-L, L],\; m\geq M),$$
thus
$$|(a_n^{(m)})^{j}-(a_n)^j|_2<\frac{\delta}{2}, \ \ \forall j\in \{0, 1, \cdots, q-1\}, n\in [-L, L],m\ge M,$$
where $a_{n}^{(m)}=((a_n^{(m)})^{j})_{j=0}^{q-1}, a_n=(a_n^j)_{j=0}^{q-1}\in K$.

We remark here that the distance $|\cdot|$ on $\mathbb{C}$ is defined by 
$|a+b\sqrt{-1}|=\max\{|a|, |b|\}$ for all $a, b\in \mathbb{R}$.

Therefore we have for all $m\ge M$ and $x\in [-N, N]$,
\begin{align*}
|g^{(m)}(x)-g(x)|&\le  \frac{1}{ C} \sum_{n\in \mathbb{Z}} \sum_{j=0}^{q-1}|(a_n^{(m)})^{j}-(a_n)^j|_2 \cdot|f( x- \frac{p}{q}(nq+j)+i)|\\
&=\frac{1}{  C} \sum_{|n|\le L} \sum_{j=0}^{q-1}|(a_n^{(m)})^{j}-(a_n)^j|_2 \cdot |f( x- \frac{p}{q}(nq+j)+i)|\\
&+\frac{1}{ C} \sum_{|n|>L} \sum_{j=0}^{q-1}|(a_n^{(m)})^{j}-(a_n)^j|_2 \cdot |f( x- \frac{p}{q}(nq+j)+i)|\\
&\le \frac{\delta}{2 C} \sum_{n\in \mathbb{Z}}\sum_{j=0}^{q-1} \frac{C_1}{1+(x-\frac{p}{q} (n+j)+i)^2} \\
&+ \frac{ q}{C}\sum_{|n|>L}\sum_{j=0}^{q-1}\frac{C_1}{1+ (x-\frac{p}{q}(n+j)+i)^2}\\
&\le \frac{\delta}{2}+\frac{\delta}{2}=\delta.
\end{align*}
This implies 
$$\lim_{m\to \infty} \left\| g^{(m)}- g \right\|_{L^{\infty}([-N, N])}=0.$$
Since $N\in \mathbb{N}$ is arbitrary, we deduce that 
$$\lim_{m\to \infty} D(g^{m}, g)=0$$
and hence
$$\lim_{m\to \infty} {\bf d}( F( (a^{(m)}, i)), F( (a, i)))=0.$$
So $F$ is continuous.

To prove the injectivity of $F$, we assume $F((a, i_1) )=F((b, i_2))$ for $(a, i_1), (b, i_2)\in Y\times \mathbb{Z}_{p}$. 
We assume that 
$$a=(a_n)_{n\in \mathbb{Z}}, b=(b_n)_{n\in \mathbb{Z}}, a_n=(a_n^j)_{j=0}^{q-1}, b_n=(b_n^j)_{j=0}^{q-1}$$
and
$$ a_n^j=(a_n^{j,1}, a_n^{j,2}),  b_n^j=(b_n^{j,1}, b_n^{j,2}).$$
It implies that 
$$\sum_{n\in \mathbb{Z}} \sum_{j=0}^{q-1}(a_n^{j,1}+a_n^{j,2}\sqrt{-1})f(x-\frac{p}{q}(nq+j)+i_1)=\sum_{n\in \mathbb{Z}} \sum_{j=0}^{q-1}(b_n^{j,1}+b_n^{j,2}\sqrt{-1})f(x-\frac{p}{q}(nq+j)+i_2)$$ 
for all $x\in\mathbb{R}$ and $i_1=i_2$. For every $m\in \mathbb{Z}$ and $j_0\in \{0, \cdots, q-1\}$, by letting $x=\frac{p}{q}(mq+ j_0)-i_1$ in the above equality we have $a_m^{j_0}=b_m^{j_0}$ and thus $a=b$. It is clear that 	$F\circ S=T\circ F$.
So $F$ is an embedding from $(Y\times\mathbb{Z}_p,S)$ to $(\mathcal{B}^{\mathbb{C}}([a+\ep_0, b]) \times\mathbb{Z}_p, T)$.

\medskip

Now we check the embeddability of $G$. Obviously, $G$ is continuous and $G \circ T=\sigma\circ G$. Since for any $(g,i)\in\mathcal{B}^{\mathbb{C}}([a+\ep_0, b]) \times\mathbb{Z}_p$, 
$$g(x)\in \mathcal{B}^{\mathbb{C}}([a+\ep_0, b]) \ \ \text{and}\ \  e^{-2\pi\sqrt{-1}cx}\in \mathcal{B}^{\mathbb{C}}([a,a+\frac{\ep_0}{2}]),$$
we have that $G$ is injective. Therefore, $G$ is an embedding from $(\mathcal{B}^{\mathbb{C}}([a+\ep_0, b]) \times\mathbb{Z}_p, T)$ to $(\mathcal{B}^{\mathbb{C}}([a, b]),\sigma)$.

\bigskip

\textbf{Part 3. Minimality of $(X,\sigma)$.} 
\medskip

Since the dynamical system $(X, \sigma)$ is topologically conjugate to $(Y\times \mathbb{Z}_p, S)$ as proved in Part 2, 
it is suffices to show that the system $(Y\times \mathbb{Z}_p, S)$ is minimal.

\medskip

Fix $\widetilde{x}=(x, i_1), \widetilde{y}=(y, i_2)\in Y\times \mathbb{Z}_p$ and $\ep>0$. We choose $L\in\mathbb{N}$ depending only on $\ep>0$ such that if two points $a=(a_n)_{n\in \mathbb{Z}}, b=(b_n)_{n\in \mathbb{Z}}$ coming from $K^{\mathbb{Z}}$ satisfy $||a_{ n }-b_{ n } ||_{\infty}< \ep/2$ ($-L\le n\le L$), then 
$$\rho\big((a,i), (b,i)\big)<\ep$$ for all $i\in \mathbb{Z}_{p}$. We take $k\in\mathbb{N}$ large enough such that $N_{k-1}>2L+1$ and $2/k<\ep$.

\textbf{Case 1. $i_1=i_2$.} Set $i_1=i_2=i_0$.
Since $y\in Y\subset Y_{k}$, by the construction of $Y_{k}$ we know that there exists $n_1\in \{0, 1, \dots, N_k-1\}$ such that 
$$(y_{n_1+i+jN_{k}})_{i=0}^{N_{k}-1}\in B_k\ \ \ (j\in \mathbb{Z}). $$
Assume that $[-L, L]\subset [n_1+(j_0-1)N_k, n_1+(j_0+1)N_k-1]$ for some $j_0\in \mathbb{Z}$. Since $x\in Y\subset Y_{k+1}$, there exists $n_2\in \{0, 1, \dots, N_{k+1}-1\}$ such that $$(x_{n_2+i+jN_{k+1}})_{i=0}^{N_{k+1}-1}\in B_{k+1}\ \ \ (j\in\mathbb{Z}).$$
By the construction of $B_{k+1}$ and noting the relation between $B_{k}$ and $B_{k+1}$, we may find $M_0\in N_{k}\mathbb{Z}$ with 
$$||(\sigma^{M_0}x)_{n}-y_n ||_{\infty}\le\frac{1}{k}<\frac{\ep}{2}\ \ \ \big(n\in\big\{n_1+j_0N_{k}, \dots, n_1+(j_0+1)N_{k}-1)\big\}\big)$$ 
and
$$(\sigma^{M_0}x)_{n}=y_{n}\ \ \ \big(n\in\big\{n_1+(j_0-1)N_k,\dots, n_1+j_0N_k-1 \big\}\big).$$
Thus, 
$$\rho((\sigma^{M_0}x,i_{0}),(y,i_{0}))<\ep.$$
Observe that $S^{M_0p}(\widetilde{x})=(\sigma^{M_0}x, i_0)$.
So 
$$\rho(S^{M_0p}(\widetilde{x}), \widetilde{y})<\ep.$$

\textbf{Case 2. $i_1\neq i_2$.} We may reduce this case to Case 1 by taking some $M_1\in\mathbb{Z}$ with 
$$S^{M_1}\widetilde{x}\in Y\times \{i_2\}.$$ 
By Case 1, there exists $M_2\in \mathbb{Z}$ such that
$$\rho(S^{M_2}( S^{M_1}\widetilde{x})),\widetilde{y})<\ep.$$
we conclude that $(Y\times \mathbb{Z}_p, S)$ is minimal.

\bigskip

\textbf{Part 4. The calculation of mean dimension of $(X, \sigma)$.} 

\medskip

We are going to prove $\text{mdim}(X, \sigma)=s$, which is equivalent to $\text{mdim}(Y\times \mathbb{Z}_p, S)=s$.  We first show $\text{mdim}(Y\times \mathbb{Z}_p, S)\geq s$. 

We fix a point $z=(z_{n})_{n\in\mathbb{Z}}\in K^{\mathbb{Z}}$ satisfying that $z|_{jN_k}^{(j+1)N_k-1}\in B_k$ for all $k\in\mathbb{N}$ and all $j\in \mathbb{Z}$. For each $k\in \N$ we define a mapping $F_{k}$ as follows:
$$F_k:B_{k}\to Y\times \mathbb{Z}_p,\ \ a=(a_n)_{n=0}^{N_k-1} \mapsto ((c_n)_{n\in \mathbb{Z}},0 ),$$
where 
$$c_n=\begin{cases}
a_n, &n\in\{0,\dots,N_{k}-1\}\\
z_n, &n\in \mathbb{Z}\setminus\{0,\dots,N_{k}-1\}.\\
\end{cases}$$
Clearly, $F_k(a)$ (where $a\in B_k$) is indeed in $Y\times \mathbb{Z}_p$ and $F_k$ is continuous. We have the following result about $F_{k}$.

\begin{claim}\label{claim:distanceinceeasing}
	The mapping $F_k:(B_{k},||\cdot||_{\infty})\to(Y\times \mathbb{Z}_p,\rho_{N_{k}p})$ is distance-increasing, i.e. for all $a,b\in B_k$
	$$||a-b||_{\infty}\le\rho_{N_kp}(F_k(a), F_{k}(b)).$$
\end{claim}
\begin{proof}[Proof of Claim \ref{claim:distanceinceeasing}.]
	Recall that 
	$$\rho_{N_kp}(\widetilde{a}, \widetilde{b})=\max_{0\le n\le N_kp-1}\rho(S^n \widetilde{a}, S^n \widetilde{b}), \ \ \widetilde{a}, \widetilde{b}\in Y\times \mathbb{Z}_{p}.$$
	Let $a, b\in B_k$. Assume that $F_k(a)=(c_1, 0)$ and $F_k(b)=(c_2, 0)$ for some $c_1, c_2\in Y$. Then for any $0\le n\le N_k-1$, we have

\begin{align*}
\rho(S^{np}F_k(a), S^{np}F_k(b) )&= D_1(\sigma^nc_1, \sigma^nc_2)\\
&\ge ||(\sigma^nc_1)_0- (\sigma^nc_2)_0||_{\infty}\\
&=||a_n-b_n||_{\infty}.
\end{align*}
Thus for any $a, b\in B_k$, we have
$$||a-b||_{\infty}\le\rho_{N_kp}(F_k(a), F_{k}(b)).$$	
It follows from the definition that for every $\ep>0$ and every $k\in\mathbb{N}$,  
$$\text{Widim}_{\ep}(B_k, ||\cdot||_{\infty}))\le\text{Widim}_{\ep}(Y\times \mathbb{Z}_p, \rho_{N_kp})$$
and hence 
$$\text{Widim}_{\ep}(Y\times \mathbb{Z}_p, \rho_{N_kp})\geq2q\cdot|x^{(k)}(N_k, \ast)|>2q\cdot(rN_{k}).$$
by Lemma \ref{Widim of [0,1]^l} and \eqref{eq:proportion_xk}.
\end{proof}
Therefore, we have
\begin{align*}
\text{mdim}(Y\times \mathbb{Z}_p,S)=&\lim\limits_{\ep\to 0}\lim\limits_{k\to \infty} \frac{\text{Widim}_{\ep}(Y\times \mathbb{Z}_p, \rho_{N_kp})}{N_kp} \\
\geq&\lim\limits_{\ep\to 0}\lim\limits_{k\to\infty} \frac{2qrN_{k}}{N_kp}\\
=&2r\frac{q}{p}=s.
\end{align*}

\medskip

Now we prove $\text{mdim}(Y\times \mathbb{Z}_p, S)\leq s$. For each $k\in\mathbb{N}$, we consider a mapping from $Y_{k}\times \mathbb{Z}_p$ to itself defined similar to $S$ (see \eqref{def:S}) and denote it still by $S$ for the saving of symbols. Since $(Y\times \mathbb{Z}_p,S)$ is a subsystem of $(Y_k\times \mathbb{Z}_p,S)$ for any $k\in \N$, it suffices to show that 
$$\text{mdim}(Y_k\times \mathbb{Z}_p,S)\le s+\epsilon_{k}\;\;\;(k\in\mathbb{N})$$
for some positive sequence $\{\epsilon_{k}\}_{k=1}^{\infty}$ tending to $0$.
\vspace{0.2cm}

Fix $k\in\mathbb{N}$ for the moment. Given $\ep>0$. Choose $L\in\mathbb{N}$ large enough such that if two points $a, b\in K^{\mathbb{Z}}$ satisfy that  $a|_{[-L, L]}=b|_{[-L, L]}$, then $\rho\big((a,i), (b,i)\big)<\ep$ for all $i\in \mathbb{Z}_p$.
\vspace{0.2cm}

For each $m\in\mathbb{N}$, define 
$$H_{k,m}: (Y_k\times \mathbb{Z}_p, \rho_{mp})\to ([0, 1]^2)^{[-L, L+m+1]}\times \mathbb{Z}_p,\;\;\;(a, i)\mapsto (a|_{[-L, L+m+1]}, i).$$

\begin{claim}\label{H injective}
	The mapping $H_{k,m}$ defined above is an $\ep$-embedding with respect to $\rho_{mp}$.
\end{claim}
\begin{proof}[Proof of Claim \ref{H injective}]
	Assume that for some $(a,i),(b,j)\in Y_{k}\times \mathbb{Z}_{p}$, $H_{k,m}(a,i)=H_{k,m}(b,j)$.
	It follows from the definition that  
	$$a|_{[-L, L+m+1]}=b|_{[-L, L+m+1]}\;\;\; \text{and}\;\;\; i=j.$$
	
	For any $0\le l\le mp-1$, we observe that
	$$S^l(a,i)=(\sigma^{l'}(a), l'')\;\;\; \text{and}\;\;\; S^l(b,i)=(\sigma^{l'}(b), l'')$$
	where $l+i=l'p+l''$ for some $l'\in\{0, 1, \dots, m+1\}$ and $l''\in\{0,1,\dots,p-1\}$. Since $a|_{[-L, L+m+1]}=b|_{[-L, L+m+1]}$, we get by the choice of $L$ that 
	$$\rho(S^l(a,i), S^l(b,j))=D_1(\sigma^{l'}(a),\sigma^{l'}(b))<\epsilon.$$
	So 
	$$\rho_{mp}((a,i),(b,j))=\max_{0\leq l\leq mp-1}\rho(S^l(a,i), S^l(b,j))<\epsilon.$$
	We conclude that $H_{k,m}$ is an $\ep$-embedding with respect $\rho_{mp}$.
\end{proof}

From Claim \ref{H injective}, we have that for all $\epsilon>0$ and all $k,m\in\mathbb{N}$,
$$\text{Widim}_{\ep}(Y_k\times \mathbb{Z}_p,\rho_{mp})\le \text{Widim}_{\ep}(H_{k,m}(Y_k\times \mathbb{Z}_p))\le\text{dim}(H_{k,m}(Y_k\times \mathbb{Z}_p))$$
and the constructions of the set $Y_{k}$ and the mapping $H_{k,m}$ make sure that 
\begin{align*}
\text{dim}(H_{k,m}(Y_k\times \mathbb{Z}_p))& \le \left(\Big\lceil\frac{2L+m+2}{N_k}\Big\rceil +1\right)\cdot\text{dim}(B_k)\\
&\leq \left(\Big\lceil\frac{2L+m+2}{N_k}\Big\rceil+1\right)\cdot\left(2N_{k}q\left(r+\frac{1}{N_{k}}\right)\right),
\end{align*}
where $\Big\lceil\frac{2L+m+2}{N_k}\Big\rceil$ is the integer part of $\frac{2L+m+2}{N_k} $.
Therefore,
\begin{align*}
\text{mdim}(Y_k\times \mathbb{Z}_p,S)&=\lim\limits_{\ep\to 0}\lim\limits_{m\to\infty}\frac{\text{Widim}_{\ep}(Y_k\times \mathbb{Z}_p,\rho_{mp})}{mp}\\
&\le\lim\limits_{\ep\to 0}\lim\limits_{m\to\infty}\frac{1}{mp}\cdot \left(\Big\lceil\frac{2L+m+2}{N_k}\Big\rceil+1\right)\cdot\left(2N_{k}q\left(r+\frac{1}{N_{k}}\right)\right)\\
&\le\frac{2q}{p}\cdot\left(r+\frac{1}{N_{k}}\right).
\end{align*}
Letting $k\to \infty$,
We get that
$$\text{mdim}(Y\times Z_p,S)\le\frac{2qr}{p}=s.$$
So far, we have proved Theorem \ref{Main Theorem 1} (\ref{1}).

\end{proof}

The remaining of this section is devoted to the proofs of Theorem \ref{Main Theorem 1} (\ref{2}) and Theorem \ref{Main Theorem 2} based on Theorem \ref{Main Theorem 1} (\ref{1}). 

\begin{proof}[Proof of Theorem \ref{Main Theorem 1} (\ref{2})]
	Fix $0\le t < 2c$ and choose $0<\ep_0<c$ with $t<2(c-\ep_0)$.
	By Theorem \ref{Main Theorem 1} (\ref{1}), there exists a minimal subsystem $(X_1,\sigma)$ of $\mathcal{B}^{\mathbb{C}}([\ep_0, c])$ with mean dimension $t$.
	Define 
	$$H: (\mathcal{B}^{\mathbb{C}}([\ep_0, c]),\sigma)\to(\mathcal{B}([-c, c]),\sigma), \ \ \ \ f\mapsto\frac{f+\bar{f}}{2}.$$
	Clearly, the mapping $H$ is continuous and $H\circ\sigma=\sigma\circ H$. The injectivity of $H$ follows from the fact that $f-g=\overline{g-f}$ $(f,g\in\mathcal{B}^{\mathbb{C}}([\ep_0, c])$ implies $\mathcal{F}(f-g)=0$ and hence $f=g$. 	
	\vspace{0.2cm}
	
	Put $X=H(X_1)$. The system $(X, \sigma)$ is topologically conjugate to $(X_1, \sigma)$ and thus is minimal with mean dimension $t$.
	
\end{proof}

\begin{proof}[Proof of Theorem \ref{Main Theorem 2}]
	Let $0\le a <1$ and fix $0\le r < a$. By Theorem \ref{Main Theorem 1} (\ref{2}), there exists a minimal subsystem $(X_1, \sigma)$ of $(\mathcal{B}([-\frac{a}{2},\frac{a}{2}]),\sigma)$ with mean dimension $r$. Define 
	$$H_1:(\mathcal{B}([-\frac{a}{2}, \frac{a}{2}]),\sigma)\to([0, 1]^{\mathbb{Z}}, \sigma), \ \ \ \ f\mapsto f|_{\mathbb{Z}}=(f(n))_{n\in \mathbb{Z}}.$$
	
	Obviously, the mapping $H_1$ is continuous and $H_1\circ\sigma=\sigma\circ H_1$.  To show the injectivity of $H_1$, we assume that 
	$$H_1(f)=H_1(g) \ \  \text{for}\ \  f,g\in \mathcal{B}([-\frac{a}{2},\frac{a}{2}]),$$ 
	then 
	$$(f-g)(n)=0 \ \  \text{for all} \ \ n\in \mathbb{Z}.$$ 
	Since $f-g\in \mathcal{B}([-\frac{a}{2},\frac{a}{2}])$, we have $f=g$ by Theorem \ref{thm:sampling}.
	
	\vspace{0.2cm}
	
	Put $X=H_1(X_1)$. Then the system $(X, \sigma)$ is a minimal subsystem of $(\mathcal{B}([-\frac{a}{2}, \frac{a}{2}]),\sigma)$ and $([0, 1]^{\mathbb{Z}}, \sigma)$ with mean dimension $r$.
	
\end{proof}


\end{document}